\documentclass[a4paper,12pt]{amsart}
\usepackage{amssymb,amsxtra,paralist,mathrsfs,upgreek,xcolor}
\usepackage[colorlinks=true,linkcolor=magenta,citecolor=cyan,urlcolor=olive]{hyperref}
\usepackage{cleveref}
\usepackage[normalem]{ulem}

\usepackage{amssymb,amsxtra,paralist,mathrsfs,upgreek,xcolor}
\usepackage[colorlinks=true,linkcolor=magenta,citecolor=cyan,urlcolor=olive]{hyperref}
\usepackage{cleveref}
\usepackage[normalem]{ulem}

\title[The weak-type $(1,1)$ estimates for wave equation]{The weak-type $(1,1)$ estimates for wave equation on $ax+b$ groups}
\author[Y. Wang and L. Yan]{Yunxiang Wang and Lixin Yan}
\address{Yunxiang Wang, Department of Mathematics, Sun Yat-sen University, Guangzhou, 510275, P. R. China}
\email{\href{mailto: wangyx677@mail.sysu.edu.cn}{wangyx677@mail.sysu.edu.cn}}
\address{Lixin Yan, Department of Mathematics, Sun Yat-sen University, Guangzhou, 510275, P. R. China}
\email{\href{mailto: mcsylx@mail.sysu.edu.cn}{mcsylx@mail.sysu.edu.cn}}

\date{\today}


\numberwithin{equation}{section}
\newtheorem{theorem}{Theorem}

\newtheorem{lemma}[theorem]{Lemma}

\newtheorem{proposition}[theorem]{Proposition}
\theoremstyle{definition}

\theoremstyle{remark}

\numberwithin{theorem}{section}

\newcommand{\R}{\mathbb{R}}
\newcommand{\e}{\mathrm{e}}
\newcommand{\Id}{\mathrm{Id}}

\newcommand{\supp}{\mathrm{supp}}

\newcommand{\arcosh}{\mathrm{arcosh}}

\newcommand{\BMO}{\mathrm{BMO}}

\newcommand{\dV}{\mathrm{d}V}
\newcommand{\bc}{\mathbf{c}}
\newcommand{\F}{\mathscr{F}}
\newcommand{\g}{\mathbf{g}}

\renewcommand{\SS}{\mathbb{S}}
\renewcommand{\i}{\mathrm{i}}
\renewcommand{\d}{\mathrm{d}}

\allowdisplaybreaks

\begin{document}
\begin{abstract}
Let $G$ be the group $\mathbb{R}_+\ltimes \mathbb{R}^n$ endowed with Riemannian symmetric space metric $d$ and the right Haar measure $\mathrm{d} \rho$ which is of $ax+b$ type, and $L$ be the positive definite distinguished left-invariant Laplacian on $G$. Let $u=u(t,\cdot)$ be the solution to $u_{tt}+Lu=0$ with initial conditions $u|_{t=0}=f$ and $u_t|_{t=0}=g$. In this article we show that for a fixed $t \in\mathbb{R}\setminus\{0\}$, 
\begin{align*}
\|u(t,\cdot)\|_{L^{1,\infty}(G,\mathrm{d}\rho)}\leq C\ (1+|t|)\big(\|(\operatorname{Id}+L)^{\alpha_0/2}f\|_{L^1(G,\mathrm{d}\rho)} 
+\|(\operatorname{Id}+L)^{\alpha_1/2}g\|_{L^1(G,\mathrm{d}\rho)}\big) 
\end{align*}
if and only if $\alpha_0\geq {n/2}$ and $\alpha_1\geq n/2-1$. Moreover, the polynomial-in-time growth $1+|t|$ in the above estimate is best possible.
\end{abstract}

\maketitle


\section{Introduction}\label{sec1}
We begin by recalling the definitions necessary for the statement of wave equation on the $ax+b$ group. Let $G$ be the Lie group $\R_+\ltimes \R^n$ endowed with the product: for all $(x,y), (x',y')\in G$,
\begin{align*}
(x,y)\cdot(x',y')=(x\,x',y+x\,y').
\end{align*}
The group $G$ is called the $(n+1)$-dimensional $ax+b$ group. Clearly, $e=(1, 0)$ is the identity of $G$. We endow $G$ with the left-invariant Riemannian metric $\g=x^{-2}(\d x^2 +\d y^2)$, and denote by $d$ the corresponding metric. This coincides with the metric on the $(n+1)$-dimensional hyperbolic space. The right Haar measure on $G$ is given by
\begin{align*}
\d \rho(x,y)=x^{-1}\,\d x\d y.
\end{align*}
The space $(G, d, \d \rho)$ is a solvable Lie group with exponential volume growth. Throughout the article, we will use the right Haar measure $\d \rho$ in notions such as $L^p(G)$ for all $1\leq p< \infty$.

A basis for the Lie algebra of left-invariant vector fields is given by
\begin{align}\label{s1e1}
X=x{\partial\over \partial x},\quad Y_j=x{\partial\over \partial y_j},\quad 1\leq j\leq n.
\end{align}
We define the distinguished left-invariant Laplacian to be the second order differential operator
\begin{align}\label{s1e2}
L=-X^2-\sum_{j=1}^nY_j^2.
\end{align}
It is essentially self-adjoint on $L^2(G)$, and thus we can use spectral calculus to define the operators $\psi(\sqrt{L})$ for Borel functions $\psi$ on $[0,\infty)$.
\smallskip

The main object we consider in this paper is the wave equation
\begin{align}\label{s1e3}
{\partial^2u\over\partial t^2}+Lu=0,\quad u|_{t=0}=f,\quad \left.{\partial u\over\partial t}\right|_{t=0}=g.
\end{align}
corresponding to $L$ on $G$. In 2000s M\"{u}ller and Thiele \cite{MuTh} showed that for $1\leq p\leq \infty$, $\alpha_0>n|1/p-1/2|$ and $\alpha_1>n|1/p-1/2|-1$, the solution $u=u(t,\cdot)$ to the wave equation \eqref{s1e3} satisfies
\begin{align}
\|u(t,\cdot)\|_{L^p(G)}\leq C_p\ \big(&(1+|t|)^{2|1/p-1/2|}\|(\Id+L)^{\alpha_0/2}f\|_{L^p(G)}\label{s1e3.1}\\
&\ \ +(1+|t|)\,\|(\Id+L)^{\alpha_1/2}g\|_{L^p(G)}\big).\notag
\end{align}
Recently, endpoint cases of estimate \eqref{s1e3.1} was obtained by the named authors of this paper \cite{WaYa}. Specifically, they showed that for $t\in\R\setminus\{0\}$ and $1<p<\infty$, the estimate \eqref{s1e3.1}
holds if and only if
\begin{align*}
\alpha_0\geq n\left|{1\over p}-{1\over 2}\right|\quad \text{and}\quad \alpha_1\geq n\left|{1\over p}-{1\over 2}\right|-1;
\end{align*}
while for $p=1$, the estimate
\begin{align}
\|u(t,\cdot)\|_{L^1(G)}\leq C\ (1+|t|)\big(&\|(\Id+L)^{\alpha_0/2}f\|_{H^1(G)}\label{s1e4} 
+\|(\Id+L)^{\alpha_1/2}g\|_{H^1(G)}\big) 
\end{align}
holds if and only if
\begin{align*}
\alpha_0\geq {n\over 2}\quad \text{and}\quad \alpha_1\geq {n\over2}-1.
\end{align*}
Here $H^1(G)$ stands for the Hardy space on the metric measure space $(G,d,\d\rho)$ (cf. \cite[Definition 2.2]{Va}).
\smallskip

The aim of this paper is to continue this line to establish the weak-type $(1,1)$ estimates for the solutions to the wave equation \eqref{s1e3}. Our main result is the following.

\begin{theorem}\label{thm1.1}
Let $t\in\R\setminus\{0\}$ and $u=u(t,(x,y))$ be the solution to equation \eqref{s1e3}. Then the following weak-type $(1,1)$ estimate
\begin{align}\label{s1e5}
\|u(t,\cdot)\|_{L^{1,\infty}(G)}\leq C\ (1+|t|)\big(&\|(\Id+L)^{\alpha_0/2}f\|_{L^1(G)} +\|(\Id+L)^{\alpha_1/2}g\|_{L^1(G)}\big) 
\end{align}
holds if and only if
\begin{align*}
\alpha_0\geq {n\over 2} \quad \mbox{and}\quad \alpha_1\geq {n\over 2}-1.
\end{align*}

Moreover, the polynomial-in-time growth $1+|t|$ in the above estimate is best possible.
\end{theorem}
The Euclidean counterpart of Theorem \ref{thm1.1} is given by Tao \cite{Ta2}. Notably, although $G$ is of exponential volume growth, one can still obtain polynomial-in-time growth estimate \eqref{s1e5}. Furthermore, in contrasts to the corresponding estimates in the Euclidean settings \cite{Mi,Pe} as well as endpoint estimates for the wave equation corresponding to the shifted Laplace--Beltrami operators on $G$ \cite{Io}, the growth rate in \eqref{s1e5} is dimension-free and is best possible. 

\smallskip

To prove the sufficiency part of Theorem \ref{thm1.1}, we follow the approach in \cite{WaYa}, which is inspired by \cite{Pe} and \cite{Ta2}. Roughly speaking, we will factorize the wave propagator $(\Id+L)^{-n/4}\e^{\i t\sqrt{L}}$ into an operator given by convolution with some weighted measure on the sphere of radius $|t|$, and a multiplier of order $0$. The proof relies heavily on the explicit asymptotic expansions of the spherical functions $\varphi_\lambda(t)$ on $G$, the spherical Fourier transform of the normalized measure on the sphere of radius $|t|$ as well as its derivative $\varphi_\lambda'(t)$ on $t$. In addition, the necessity part of Theorem \ref{thm1.1} follows directly from the one of \cite[Theorem 1.1]{WaYa}. For the sharpness of the polynomial growth $1+|t|$ in \eqref{s1e5}, we construct a function $q_t$ and use the unit speed propagation property as well as the uniqueness of the solution to show that for all $\beta_0,\beta_1\in\R$,
if 
\begin{align*}
\|u(t,\cdot)\|_{L^{1,\infty}(G)}\leq C(t)\,\big(\|(\Id+L)^{\beta_0}f\|_{L^1(G)}+\|(\Id+L)^{\beta_1}g\|_{L^1(G)}\big)
\end{align*}
holds for $(f,g)=(q_t,0)$ and $(f,g)=(0,q_t)$, then we must have $C(t)\geq C$ and $C(t)\geq C|t|$, respectively. 

\smallskip

The paper is organized as follows. In Section \ref{sec2} we present basic facts about the Fourier theory as well as the singular integral theory on $G$. In Section \ref{sec3} we establish one proposition, which plays an important role in the proof of the sufficiency part of Theorem \ref{thm1.1}. The proof of our main result, Theorem \ref{thm1.1}, will given in 
Section \ref{sec5}.

\bigskip


\section{Preliminaries}\label{sec2}
In this section we review the analysis on the $ax+b$ group $G=\R_+\ltimes \R^n$. For the detail, we refer the reader to \cite{Br, Da, Gr, HeSt, Pe2, Va}.
\subsection{\texorpdfstring{$G$}{G} as the hyperbolic space}
The space $G$ endowed with the left-invariant Riemannian metric given by
\begin{align}\label{riem}
\g=x^{-2}(\d x^2+\d y^2)
\end{align}
is the $(n+1)$-dimensional hyperbolic space. The geodesic distance $d$ is a left-invariant distance on $G$. For all $(x,y)$ in $G$
\begin{align}\label{e2.2}
d((x,y),e)=\arcosh {x^2+1+|y|^2\over 2x},
\end{align}
where $e=(1,0)$ is the identity on $G$ and $|\cdot|$ is the $n$-dimensional Euclidean norm. For simplicity in the sequel we denote $R(x,y)=d((x,y),e)$ for $(x,y)\in G$. The measure $\dV$ induced by the Riemannian metric $\g$ is a left Haar measure on $G$ and is given by
\begin{align*}
\dV(x,y)=x^{-(n+1)}\,\d x\d y_1\dots\d y_n.
\end{align*}
In particular,
\begin{align}\label{e2.4}
\dV (x,y)=\delta(x,y)\,\d \rho(x,y),
\end{align}
where the modular function is given by 
\begin{eqnarray}\label{mm}
\delta(x,y)=x^{-n}.
\end{eqnarray}
\smallskip

Following \cite[p. 21 and p. 136]{Pe2} and \cite[Exercise 3.18]{Gr}, we can adopt polar coordinates 
\begin{align}\label{polcor}
(x,y)=g(r,\omega)=(x(r,\omega),y(r,\omega)), \quad (r,\omega)\in [0,\infty)\times \SS^n,
\end{align}
where $r=R(x,y)\geq 0$, $\omega=(\omega',\omega_{n+1})\in \SS^n\subset \R^{1+n}$,
\begin{align}\label{epol}
x(r,\omega)={1\over(\omega_{n+1}+\coth r)\, \sinh r}\quad \mbox{and} \quad y(r,\omega)={\omega'\over \omega_{n+1}+\coth r},
\end{align}

We now introduce a simple lemma that will be repeatedly used throughout this article.
\begin{lemma}[{\cite[Lemma 2.1]{WaYa}}]\label{lem2.1}
Suppose $r>0$. Then
\begin{align*}
\int_{\SS^n}\delta^{-1/2}(g(r,\omega))\,\d\omega+\int_{\SS^n}\Big|\partial_r[\delta^{-1/2}(g(r,\omega))]\Big|\,\d\omega\leq C\,(1+r)\,\e^{-nr/2}.
\end{align*}
\end{lemma}

\subsection{Spherical Fourier analysis on \texorpdfstring{$G$}{G}}
For a radial function $\kappa$, one can define its spherical Fourier transform by
\begin{align}\label{e2.9}
\F(\kappa)(\lambda)=\nu_n\int_0^\infty \kappa(r)\, \varphi_\lambda(r)\,(\sinh r)^n\,\d r,
\end{align}
where $\nu_n$ is the surface area of the unit sphere $\SS^n$ and $\varphi_{\lambda}$ denotes the elementary spherical function given by
\begin{align}\label{e2.10}
\varphi_\lambda(t)={2^{n/2-1}\Gamma((n+1)/2)\over \sqrt{\uppi}\,\Gamma(n/2)} (\sinh t)^{1-n}\int_{-t}^t\e^{\i \lambda s}(\cosh t-\cosh s)^{n/2-1}\,\d s,
\end{align}
see \cite[(3.1.4)]{Br}. We also have the inverse formula of the above Fourier transform:
\begin{align*}
\kappa(r)={2^{n-2}\Gamma((n+1)/2\over\uppi^{(n+3)/2}}\int_0^\infty \F (\kappa)(\lambda)\,\varphi_\lambda(r)\,|\bc(\lambda)|^{-2}\,\d\lambda,
\end{align*}
where $\bc$ denotes the Harish-Chandra function given by
\begin{align}\label{e2.11}
\bc(\lambda)={2^{n-1}\Gamma((n+1)/2))\over\sqrt{\uppi}}{\Gamma(\i\lambda)\over \Gamma(n/2+\i\lambda)},
\end{align}
see \cite[(3.1.3)]{Br}.
\smallskip

Now we introduce the spectral multipliers on $G$. The distinguished Laplacian $L$ given by \eqref{s1e2} has a special relationship with the (positive definite) Laplace--Beltrami operator $\Delta $ on $G$. Indeed, in view of \cite[p. 177]{Da}, $\Delta$ has a continuous spectrum $[n^2/4,\infty)$ on $L^2(G,\dV)$. If we let $\Delta_n$ be the shifted Laplace--Beltrami operator $\Delta-n^2\Id/4$, it follows from \cite[Proposition 2]{As} that
\begin{align*}
L(f)(x,y)=\delta^{1/2}(x,y)\, \Delta_n(\delta^{-1/2} f)(x,y).
\end{align*}
This combined with the spectral theorem gives
\begin{align}\label{e3.3}
\psi(\sqrt{L})(f)(x,y)=\delta^{1/2}(x,y)\,\psi(\sqrt{\Delta_n})(\delta^{-1/2}f)(x,y).
\end{align}

Let $k_\psi$ and $\kappa_\psi$ denote the convolution kernel of $\psi(\sqrt{L})$ and $\psi(\sqrt{\Delta_n})$, respectively, i.e.,
\begin{align*}
\psi(\sqrt{L})(f)=f*k_\psi\quad\mbox{and}\quad \psi(\sqrt{\Delta_n})(f)=f*\kappa_\psi,
\end{align*}
where ``$*$'' denotes the convolution on $G$ defined by
\begin{align}\label{con}
f*k(g)=\int_G f(h)\,k(h^{-1} g)\,\delta(h)\,\d\rho(h).
\end{align}
Then in view of the following lemma, the Fourier transform of $\kappa_\psi$ is the multiplier function $\psi$.
\begin{lemma}\label{lem3.1}
Let $\psi$ be a Borel function on $\R_+$. Then $\kappa_\psi$ is radial and is given by
\begin{align*}
\kappa_\psi(x,y)={2^{n-2}\Gamma((n+1)/2)\over\uppi^{(n+3)/2}}\int_0^\infty \psi(\lambda)\,\varphi_\lambda(R(x,y))\,|\bc(\lambda)|^{-2}\,\d \lambda.
\end{align*}
\end{lemma}
\begin{proof}
See \cite[Section 4]{Br}.
\end{proof}

In view of Lemma \ref{lem3.1} and the relationship \eqref{e3.3} between $\psi(\sqrt{\Delta_n})$ and $\psi(\sqrt{L})$, we can derive the expression of $k_\psi$ as follows:
\begin{align}
&k_\psi(x,y)\label{e3.4}\\
=&\delta^{1/2}(x,y)\,\kappa_\psi(x,y)\notag\\
=&{2^{n-2}\Gamma((n+1)/2)\over\uppi^{(n+3)/2}}\delta^{1/2}(x,y)\int_0^\infty \psi(\lambda)\,\varphi_\lambda(R(x,y))\,|\bc(\lambda)|^{-2}\,\d \lambda.\notag
\end{align}
\subsection{Singular integrals on \texorpdfstring{$G$}{G}}
The following singular integral theorem and multiplier theorem are due to Hebisch.
\begin{lemma}\label{lem2.5}
Let $T$ be a linear operator which is bounded on $L^2(G)$ and admits a locally integrable Schwartz kernel $K$ with respect to $\d\rho$ off the diagonal which satisfies the \emph{H\"{o}rmonder smoothness condition}: for some $\varepsilon>0$ small enough,
\begin{align*}
\sup_{h,h'\in G}\int_{\{g:d(g,h),d(g,h')>\varepsilon d(h,h')\}}|K(g,h)-K(g,h')|\,\d \rho(g)<\infty.
\end{align*}
Then $T$ extends to a bounded operator from $L^1(G)$ to $L^{1,\infty}(G)$.
\end{lemma}
\begin{proof}
This follows directly from \cite[Theorem 1.2, Remark 1.4 and Lemma 5.1]{HeSt}.
\end{proof}
\begin{lemma}[{\cite[Theorem 2.4]{HeSt}}]\label{lem2.6}
Suppose that $s_0>3/2$ and $s_\infty>\max\{3/2,(n+1)/2\}$. Let $\beta\in C^\infty(\R)$ with $\supp \beta\subset [1/4,4]$ such that
\begin{align*}
\sum_{j=-\infty}^\infty \beta(2^{-j}\lambda)\equiv 1\quad\mbox{for all }\lambda \in \R_+.
\end{align*}
If $m\in L^\infty(\R)$ satisfies
\begin{align*}
\sup_{0<\tau\leq 1}\|m(\tau\cdot)\,\beta(\cdot)\|_{L^2_{s_0}(\R)}+\sup_{\tau>1}\|m(\tau\cdot)\,\beta(\cdot)\|_{L^2_{s_\infty}(\R)}<\infty,
\end{align*}
then the operator $m(\sqrt{L})$ is bounded from $L^1(G)$ to $L^{1,\infty}(G)$.
\end{lemma}
\bigskip


\section{A useful result}\label{sec3}
In this section, we establish the following proposition, which serves as a key ingredient in the proof of the sufficiency part of Theorem \ref{thm1.1}.

\begin{proposition}\label{thm3.2}
Let $m$ be an even function with $m\in S^{-1}$. Then for $j=1,2,\dots,n$, $m(\sqrt{L})X^*$ and $m(\sqrt{L})Y_j^*$ are bounded from $L^1(G)$ to $L^{1,\infty}(G)$, where $X,Y_1,\dots,Y_j$ are the left-invariant vector fields given by \eqref{s1e1}.
\end{proposition}

To show Proposition \ref{thm3.2}, we need the following convenient integral representation for the convolution kernels of the spectral multipliers.
\begin{lemma}\label{lem3.3}
Suppose $m$ is an even Borel function defined on $\R_+$. Then the convolution kenrel $k_m$ of $m(\sqrt{L})$ given by \eqref{e3.4}, has the following convenient integral representation:
\begin{align}\label{e3.5}
k_m(x,y)=\delta^{1/2}(x,y)\int_{\R} m(\lambda) \, F_{R(x,y)}(\lambda)\,\lambda\,\d \lambda
\end{align}
for
\begin{align*}
F_r(\lambda):=C_\ell\int_r^\infty D^\ell_{\mathrm{sh},s}\big[\e^{\i\lambda s}\big]
\,(\cosh s-\cosh r)^{\ell-n/2}\,\d s,
\end{align*}
where $\ell$ is any integer satisfying $\ell>n/2-1$ and we have written $D^\ell_{\mathrm{sh},s}$ for the $\ell$-th power of $D_{\mathrm{sh}}: g(s)\mapsto (\d/\d s)[g(s)/\sinh s]$ acting on the variable $s$.

In addition, for $i=0,1,2$, $F^{(i)}_r(\lambda):=\partial^i_r[F_r(\lambda)]$ satisfy the following asymptotic expansions:
\begin{asparaenum}[\rm (i)]
\item For $r>1$ and $i=0,1,2$,
\begin{align*}
F^{(i)}_r(\lambda)=\e^{\i\lambda r}\e^{-nr/2}b_{n/2-1+i}(\lambda,r),
\end{align*}
where $b_\sigma(\cdot,r)\in S^\sigma$ uniformly in $r>1$.
\item For $0<r\leq 1$ and $i=1,2$,
\begin{align*}
F^{(0)}_r=\e^{\i\lambda r}\times\begin{cases}
r^{-1/2}b_{-1/2}(\lambda,r)&\text{if }n=1,\\
r^{-n/2}b_{n/2-1}(\lambda,r)+r^{1-n}b_0(\lambda,r)&\text{if }n\geq 2,
\end{cases}
\end{align*}
\begin{align*}
F^{(i)}_r(\lambda)=\e^{\i\lambda r}\big(r^{-n/2+1}b_{n/2+i}(\lambda,r)+r^{-n+1-i}b_0(\lambda,r)\big),
\end{align*}
where $b_\sigma(\cdot,r)\in S^\sigma$ uniformly in $0<r\leq 1$.
\end{asparaenum}
\end{lemma}
\begin{proof}
The case $i=0$ is \cite[Propositions 5.2 and 5.7]{MuTh}. To show the case $i=1,2$ one just follows line-by-line modifications of the proof of \cite[Propositions 5.2 and 5.7]{MuTh}.
\end{proof}
\begin{proof}[Proof of Proposition \ref{thm3.2}]
Let us first show the weak-type $(1,1)$ of the operator $m(\sqrt{L})X^*$. Denote the convolution kernels of $m(\sqrt{L})X^*$ and $X\overline m(\sqrt{L})$ by $k$ and $\widetilde{k}$, respectively. Then in view of \eqref{e3.5} and the fact that
\begin{align*}
XR(x,y)={1\over \sinh R(x,y)}{x^2-1-|y|^2\over 2x},
\end{align*}
we have
\begin{align*}
\widetilde{k}(x,y)=&X\!\left[\delta^{1/2}\int_{\R}\overline m(\lambda)\,F_{R(\cdot)}(\lambda)\,\lambda\,\d\lambda\right]\!(x,y)\\
=&-{n\over 2}\delta^{1/2}(x,y) \int_{\R}\overline m(\lambda)\,F_{R(x,y)}(\lambda)\,\lambda\,\d\lambda\\
&\quad+{\delta^{1/2}(x,y)\over\sinh R(x,y)}{x^2-1-|y|^2\over 2x}\int_{\R}\overline m(\lambda)\,F^{(1)}_{R(x,y)}(\lambda)\,\lambda\,\d\lambda.
\end{align*}
Hence
\begin{align}\label{eA.1}
k(x,y)=&\big(\widetilde{k}\big)^*(x,y)\\
=&-{n\over 2}\delta^{1/2}(x,y) \int_{\R}\overline m(\lambda)\,F_{R(x,y)}(\lambda)\,\lambda\,\d\lambda\notag\\
&\quad-{\delta^{1/2}(x,y)\over \sinh R(x,y)}{x^2-1+|y|^2\over 2x}\int_{\R}\overline m(\lambda)\,F^{(1)}_{R(x,y)}(\lambda)\,\lambda\,\d\lambda\notag\\
=&\widetilde{k}(x,y)-\underbrace{{\delta^{1/2}(x,y)\over \sinh R(x,y)}{x^2-1\over x}\int_{\R}\overline m(\lambda)\,F^{(1)}_{R(x,y)}(\lambda)\,\lambda\,\d\lambda}_{q(x,y)},\notag
\end{align}
where the superscript ``${}^*$'' stands for the $L^1$-isometric involution on $G$ given by
\begin{align*}
f^*(x,y)=\delta(x,y)\,\overline{f((x,y)^{-1})}=\delta(x,y)\,\overline{f(1/x,-y/x)}.
\end{align*}
\eqref{eA.1} indicates that
\begin{align*}
m(\sqrt{L})X^*(f)=Xm(\sqrt{L})(f)-f*q.
\end{align*}
However, by \cite[(A.4)]{WaYa} and Lemma \ref{lem2.5}, the operator $X\overline m(\sqrt{L})$ is of weak-type $(1,1)$. Therefore to show $m(\sqrt{L})X^*$ is bounded from $L^1(G)$ to $L^{1,\infty}(G)$ it suffices to show the operator given by convolution with the kernel $q$ is also of weak-type $(1,1)$. The idea is to verify that its Schwartz kernel $Q(g,h)=\delta(h)\,q(h^{-1}g)$ satisfies the H\"{o}rmonder smoothness condition in Lemma \ref{lem2.5}.

For two points $g=(x,y)$ and $h=(x',y')$ in $G$,
\begin{align*}
Q(g,h)&=\delta(h)\,q\!\left({x\over x'},{y-y'\over x'}\right)\\
&={\delta^{1/2}(gh)\over \sinh d(g,h)}{x^2-(x')^2\over x\,x'}\int_{\R}\overline m(\lambda)\,F^{(1)}_{d(g,h)}(\lambda)\,\lambda\,\d\lambda.
\end{align*}
Hence,
\begin{align*}
X[Q(g,\cdot)](h)=&-{n\over 2}{\delta^{1/2}(gh)\over \sinh d(g,h)}{x^2-(x')^2\over x\,x'}\int_{\R}\overline m(\lambda)\,F^{(1)}_{d(g,h)}(\lambda)\,\lambda\,\d\lambda\\
&\quad-\delta^{1/2}(gh){\cosh d(g,h)\over (\sinh d(g,h))^2}X[d(g,\cdot)](h)\\
&\qquad\qquad\times{x^2-(x')^2\over x\,x'}\int_{\R}\overline m(\lambda)\,F^{(1)}_{d(g,h)}(\lambda)\,\lambda\,\d\lambda\\
&\quad-{\delta^{1/2}(gh)\over \sinh d(g,h)}{x^2+(x')^2\over x\,x'}\int_{\R}\overline m(\lambda)\,F^{(1)}_{d(g,h)}(\lambda)\,\lambda\,\d\lambda\\
&\quad+{\delta^{1/2}(gh)\over \sinh d(g,h)}X[d(g,\cdot)](h)\\
&\qquad\qquad\times{x^2-(x')^2\over x\,x'}\int_{\R}\overline m(\lambda)\,F^{(2)}_{d(g,h)}(\lambda)\,\lambda\,\d\lambda,
\end{align*}
while for $j=1,2,\dots,n$
\begin{align*}
Y_j[Q(g,\cdot)](h)=&-\delta^{1/2}(gh){\cosh d(g,h)\over (\sinh d(g,h))^2}Y_j[d(g,\cdot)](h)\\
&\qquad\qquad\times{x^2-(x')^2\over x\,x'}\int_{\R}\overline m(\lambda)\,F^{(1)}_{d(g,h)}(\lambda)\,\lambda\,\d\lambda\\
&\quad+{\delta^{1/2}(gh)\over \sinh d(g,h)}Y_j[d(g,\cdot)](h)\\
&\qquad\qquad\times{x^2-(x')^2\over x\,x'}\int_{\R}\overline m(\lambda)\,F^{(2)}_{d(g,h)}(\lambda)\,\lambda\,\d\lambda.
\end{align*}
Note that $|X[d(g,\cdot)](h)|\leq 1$, $|Y_j[d(g,\cdot)](h)|\leq 1$ for $j=1,2,\dots,n$,
\begin{align*}
{x^2-(x')^2\over x\,x'}\leq C{d(g,h)\over 1+d(g,h)}\e^{d(g,h)},\quad{x^2+(x')^2\over x\,x'}\leq C \e^{d(g,h)}.
\end{align*}
Thus in view of Lemma \ref{lem3.3} and he classical estimates on Fourier transform of symbols \cite[p. 241]{St}, for all $N>0$,
\begin{align*}
|\nabla [Q(g,\cdot)](h)|_\g&\sim |X[Q(g,\cdot)](h)|+\sum_{j=1}^n|Y_j[Q(g,\cdot)](h)|\\
&\leq C_N\,\delta^{1/2}(gh)\times\begin{cases}
r^{-n-2}&\text{for }0<d(g,h)\leq 1,\\
r^{-N}\e^{-nr/2}&\text{for }d(g,h)>1.
\end{cases}
\end{align*}
Then the desired result that the operator given by convolution with the kernel $q$ is of weak-type $(1,1)$ follows from Lemma \eqref{lem2.5} together with a similar argument in the proof of \cite[Lemma 2.11]{WaYa}.

Now for $j=1,2,\dots,n$, we make line-by-line modifications of the above proof of $m(\sqrt{L})X^*$ to obtain that $m(\sqrt{L})Y_j^*$ are all bounded from $L^1(G)$ to $L^{1,\infty}(G)$, and is therefore omitted. See also \cite[Theorem 1.2]{Ma}.
The proof of Proposition~\ref{thm3.2} is complete.
\end{proof}


\section{Proof of Theorem \ref{thm1.1}}\label{sec5}
As is mentioned in Introduction, the proof of Theorem \ref{thm1.1} relies heavily on the explicit asymptotic expansions of $\varphi_\lambda(t)$ and $\varphi_\lambda'(t)$ for large $\lambda$, which we now state.
\begin{lemma}\label{prop4.1}
Suppose $t>0$ and $\lambda\geq \max\{1,t^{-1}\}$.
\begin{asparaenum}[\rm (a)]
\item For $\varphi_\lambda(t)$ we have
\begin{align*}
\varphi_\lambda(t)={2^{n/2}\Gamma((n+1)/2)\over\sqrt{\uppi}}(\sinh t)^{-n/2} \lambda^{-n/2}\cos\!\left(t\lambda-{n\uppi\over 4}\right)+\mathfrak{r}_t(\lambda),
\end{align*}
where
\begin{align*}
\mathfrak{r}_t(\lambda)=\begin{cases}
\e^{\i t\lambda}a_t(t\lambda)+\e^{-\i t\lambda}a_t(-t\lambda)&\mbox{if }0<t\leq 1 \mbox{ and }t\lambda\geq 1,\\
\e^{-nt/2}\left(\e^{\i t\lambda}a_t(\lambda)+\e^{-\i t\lambda}a_t(-\lambda)\right)&\mbox{if }t>1 \mbox{ and }\lambda\geq 1
\end{cases}
\end{align*}
for some $a_t\in S^{-1-n/2}$ uniformly in $t>0$.
\item For $\varphi_\lambda'(t)$ we have
\begin{align*}
\varphi_\lambda'(t)=-{2^{n/2}\Gamma((n+1)/2)\over\sqrt{\uppi}}(\sinh t)^{-n/2} \lambda^{1-n/2}\sin\!\left(t\lambda-{n\uppi\over 4}\right)+\mathfrak{s}_t(\lambda),
\end{align*}
where
\begin{align*}
\mathfrak{s}_t(\lambda)=\begin{cases}
t^{-1}\left(\e^{\i t\lambda}b_t(t\lambda)+\e^{-\i t\lambda}b_t(-t\lambda)\right)&\mbox{if }0<t\leq 1 \mbox{ and }t\lambda\geq 1,\\
\e^{-nt/2}\left(\e^{\i t\lambda}b_t(\lambda)+\e^{-\i t\lambda}b_t(-\lambda)\right)&\mbox{if }t>1 \mbox{ and }\lambda\geq 1
\end{cases}
\end{align*}
for some $b_t\in S^{-n/2}$ uniformly in $t>0$.
\end{asparaenum}
\end{lemma}

\begin{proof}
For the proof, we refer it to {\cite[Proposition 3.1]{WaYa}}.
\end{proof}

Further, we need the connection between the spherical function $\varphi_\lambda(t)$ and the spherical average operator.
\begin{lemma}\label{lem4.1}
Let $\d\sigma_t$ be the normalized spherical measure of radius $t$ on $G$, i.e.
\begin{align*}
\int_G f\,\d\sigma_t={1\over\nu_n}\int_{\SS^n}f(g(t,\omega))\,\d\omega,
\end{align*}
where $g(r,\omega)$ are the polar coordinates as in \eqref{polcor}. Then
\begin{align*}
\varphi_{\sqrt{L}}(t)(f)(x,y)={1\over\nu_n}\int_{\SS^n}\delta^{-1/2}(g(t,\omega))\, f((x,y)\cdot g(t,\omega))\,\d \omega.
\end{align*}
\end{lemma}

\begin{proof}
For the proof, we refer it to \cite[(3.1)]{WaYa}.
\end{proof}

\smallskip

We now proceed to the proof of Theorem \ref{thm1.1}. To show its sufficiency part, it suffices to show the following more general result.
\begin{proposition}\label{prop5.2}
Let $t\in\R$ and $m\in S^{-\alpha}$ be an even symbol.
\begin{asparaenum}[\rm (a)]
\item \label{prop5.2(a)}If $\alpha=n/2$, then
\begin{align*}
\|m(\sqrt{L})\,\cos(t\sqrt{L})\|_{L^1(G)\to L^{1,\infty}(G)}\leq C\,(1+|t|).
\end{align*}
for some constant $C>0$ independent of $t$.
\item \label{prop5.2(b)}If $\alpha=n/2-1$, then
\begin{align*}
\left\|m(\sqrt{L})\,{\sin(t\sqrt{L})\over \sqrt{L}}\right\|_{L^1(G)\to L^{1,\infty}(G)}\leq C\,(1+|t|). 
\end{align*}
for some constant $C>0$ independent of $t$.
\end{asparaenum}
\end{proposition}
\begin{proof}
Without loss of generality we will assume that $t>0$. To prove (\ref{prop5.2(a)}), we suppose that $m\in S^{-n/2}$ is an even function, and consider two cases: $t>1$ and $0<t\leq 1$. 
\smallskip

\begin{asparaenum}[\bf {Case} 1:]
\item \label{thm3.3case1}$t>1$. Let $\eta\in C^\infty_c(\R)$ be an even function satisfying $\eta\equiv 1$ on $[-1,1]$. It follows from \cite[Theorem 8.1 (a)]{MuTh} that
\begin{align*}
\|\eta(\sqrt{L})\,m(\sqrt{L})\,\cos(t\sqrt{L})\|_{L^1(G)\to L^1(G)} \leq C\,t
\end{align*}
for some constant $C>0$ independent of $t$. Therefore it remains to show that
\begin{align}\label{e5.8}
\|(1-\eta(\sqrt{L}))\,m(\sqrt{L})\, \cos(t\sqrt{L})\|_{L^1(G)\to L^{1,\infty}(G)} \leq C\,t. 
\end{align}

To prove \eqref{e5.8}, the crucial observation is that in view of Lemma \ref{prop4.1}, on $\supp(1-\eta)$,
\begin{align}
\cos(t\lambda)&=\cos{n\uppi\over4}\,\cos\!\left(t\lambda-{n\uppi\over4}\right)-\sin{n\uppi\over4}\,\sin\!\left(t\lambda-{n\uppi\over4}\right)\label{e5.9}\\
&=c_1\,(\sinh t)^{n/2}\lambda^{n/2}\varphi_\lambda(t)+c_2\,(\sinh t)^{n/2}\lambda^{n/2-1}\varphi_\lambda'(t)\notag\\
&\qquad+\underbrace{\e^{\i t\lambda}\mathfrak{a}_{1,t}(\lambda)+\e^{-\i t\lambda}\mathfrak{a}_{2,t}(\lambda)}_{\mathfrak{r}_t(\lambda)} \nonumber
\end{align}
with $\mathfrak{a}_{1,t}(\lambda), \mathfrak{a}_{2,t}(\lambda)\in S^{-1}$ uniformly in $t\geq 1$. In addition, from \cite[(4.3)]{WaYa} we have
\begin{align}\label{e5.10}
\|(1-\eta(\sqrt{L}))\,m(\sqrt{L})\,\mathfrak{r}_t(\sqrt{L})\|_{L^1(G)\to L^1(G)}\leq C\, t.
\end{align}

Now we consider the two terms $(\sinh t)^{n/2}\lambda^{n/2}\varphi_\lambda(t) 
$ and $ (\sinh t)^{n/2}\lambda^{n/2-1}\varphi_\lambda'(t)$ in \eqref{e5.9}. If we set
\begin{align*}
\mathfrak{m}_{1,t}(\lambda):=(\sinh t)^{n/2}\mathfrak{m}_1(\lambda)\,\varphi_\lambda(t),\quad \mathfrak{m}_{2,t}(\lambda):=(\sinh t)^{n/2}\mathfrak{m}_2(\lambda)\,\varphi_\lambda'(t)
\end{align*}
with
\begin{align*}
\mathfrak{m}_1(\lambda):=(1-\eta(\lambda))\, m(\lambda)\,\lambda^{n/2}\in S^0
\end{align*}
and
\begin{align*}
\mathfrak{m}_2(\lambda):=(1-\eta(\lambda))\, m(\lambda)\,\lambda^{n/2-1}\in S^{-1},
\end{align*}
in view of \eqref{e5.9} and \eqref{e5.10}, the proof of \eqref{e5.8} reduces to showing that
\begin{align}\label{e5.11}
\|\mathfrak{m}_{1,t}(\sqrt{L})\|_{L^1(G)\to L^{1,\infty}(G)}\leq C\, t,\quad \|\mathfrak{m}_{2,t}(\sqrt{L})\|_{L^1(G)\to L^{1,\infty}(G)}\leq C\, t.
\end{align}

We now show the first estimate in \eqref{e5.11}. In view of Lemma \ref{lem2.6}, the operator $\mathfrak{m}_1(\sqrt{L})$ is bounded from $L^1(G)$ to $L^{1,\infty}(G)$. In addition, by Lemmas \ref{lem4.1} and \ref{lem2.1}, we have
\begin{align*}
\|(\sinh t)^{n/2}\varphi_{\sqrt{L}}(t)\|_{L^1(G)\to L^1(G)}\leq C\,\e^{nt/2}\int_{\SS^n}\delta^{-1/2}(g(t,\omega))\,\d\omega\leq C\, t.
\end{align*}
Therefore,
\begin{align*}
\|\mathfrak{m}_{1,t}(\sqrt{L})\|_{L^1(G)\to L^{1,\infty}(G)}&\leq \|\mathfrak{m}_1(\sqrt{L})\|_{L^1(G)\to L^{1,\infty}(G)}\\
&\qquad\times\|(\sinh t)^{n/2}\varphi_{\sqrt{L}}(t)\|_{L^1(G)\to L^1(G)}\leq C\, t,
\end{align*}
and the first estimate in \eqref{e5.11} is valid.

Now we verify the second estimate in \eqref{e5.11}. It follows from Lemma \ref{lem4.1} that
\begin{align*}
\varphi_{\sqrt{L}}'(t)=A_t+B_{0,t}X+\sum_{j=1}^nB_{j,t}Y_j,
\end{align*}
where
\begin{align}\label{e3.11.-2}
A_t(f)(g):={1\over\nu_n}\int_{\SS^n}\partial_t[\delta^{-1/2}(g(t,\omega))]\,f(g\cdot g(t,\omega))\,\d\omega,
\end{align}
and for each $j=0,1,\dots,n$,
\begin{align}\label{e3.11.-1}
B_{j,t}(f)(g):={1\over\nu_n}\int_{\SS^n}\delta^{-1/2}(g(t,\omega))\,f(g\cdot g(t,\omega))\,v_{j,g}(t,\omega)\,\d \omega,
\end{align}
and $v_{j,g}(\cdot,\omega)$ is the $j$-th component of the velocity of the geodesic $g\cdot g(\cdot,\omega)$, i.e.,
\begin{align*}
\partial_t[g\cdot g(t,\omega)]=v_{0,g}(t,\omega)\,X(g\cdot g(t,\omega))+\sum_{j=1}^nv_{j,g}(t,\omega)\,Y_j(g\cdot g(t,\omega)).
\end{align*}
This means that for all $g\in G$, $t>0$ and $\omega\in \SS^n$,
\begin{align}\label{e3.11.1}
\sum_{j=0}^n|v_{j,g}(t,\omega)|^2\equiv 1
\end{align}
In addition, it follows from \eqref{e2.10} that for $\lambda,t\in\R$, the soherical function $\varphi_\lambda(t)$ is real-valued, whence
\begin{align}\label{e3.11.0}
\varphi_{\sqrt{L}}'(t)=\left(\varphi_{\sqrt{L}}'(t)\right)^*=A_t^*+X^*B_{0,t}^*+\sum_{j=1}^nY_j^*B_{j,t}^*.
\end{align}
By \eqref{e3.11.-2} and \eqref{e3.11.-1},
\begin{align*}
A_t^*(f)(g)={1\over\nu_n}\int_{\SS^n}\partial_t[\delta^{-1/2}(g(t,\omega))]\,f(g\cdot (g(t,\omega))^{-1})\,\d\omega,
\end{align*}
while for $j=0,1,\dots n$,
\begin{align*}
B_{j,t}^*(f)(g)={1\over\nu_n}\int_{\SS^n}\delta^{-1/2}(g(t,\omega))\,f(g\cdot (g(t,\omega))^{-1})\,v_{j,g}(t,\omega)\,\d \omega.
\end{align*}
According to \eqref{e3.11.1}, Lemmas \ref{lem2.1} and \ref{lem4.1},
\begin{align}\label{e3.11.2}
\|A_r^*\|_{1\to1}+\sum_{j=0}^n\|B_{j,r}^*\|_{1\to1}\leq C\,(1+r)\,\e^{-nr/2}\quad \text{for }r>0.
\end{align}
In view of \eqref{e3.11.0}, we can rewrite
\begin{align}
\mathfrak{m}_{2,t}(\sqrt{L})&=(\sinh t)^{n/2}\mathfrak{m}_2(\sqrt{L})\,\varphi_{\sqrt{L}}'(t)\label{e3.11.3}\\
&=(\sinh t)^{n/2}\mathfrak{m}_2(\sqrt{L})\,A_t^*+(\sinh t)^{n/2}\mathfrak{m}_2(\sqrt{L})X^*\,B_{0,t}^*\notag\\
&\qquad+\sum_{j=1}^n(\sinh t)^{n/2}\mathfrak{m}_2(\sqrt{L})Y_j^*\,B_{j,t}^*.\notag
\end{align}
This, together with \eqref{e3.11.2}, Lemmas \ref{lem2.6} and \ref{thm3.2}, yields that
\begin{align*}
\|\mathfrak{m}_{2,t}(\sqrt{L})\|_{L^1(G)\to L^{1,\infty}(G)}\leq C\,t,
\end{align*}
and the second estimate in \eqref{e5.11} holds. From \eqref{e5.11}, we conclude the proof of (\ref{prop5.2(a)}) for $t>1$ in {\bf Case} \ref{thm3.3case1}.
\smallskip
\item \label{thm3.3case2}$0<t\leq 1$. Let $\eta\in C^\infty_c(\R)$ be the cutoff function as above. We set $m_{t,0}(\lambda):=\eta(t\lambda)\,m(\lambda)\,\cos(t\lambda)$. Note that $m_{t,0}\in S^{-n/2}\subset S^0$ uniformly in $0<t\leq 1$. We apply Lemma \ref{lem2.6} to obtain
\begin{align*}
\|m_{t,0}(\sqrt{L})\|_{L^1(G)\to L^{1,\infty}(G)}\leq C.
\end{align*}
Therefore it remains to show that
\begin{align}\label{e5.4}
\|(1-\eta(t\sqrt{L}))\,m(\sqrt{L})\,\cos(t\sqrt{L})\|_{L^1(G)\to L^{1,\infty}(G)}\leq C.
\end{align}

In view of Lemma \ref{prop4.1}, on $\supp (1-\eta(t\cdot))$,
\begin{align}
\cos(t\lambda)&=\cos{n\uppi\over4}\,\cos\!\left(t\lambda-{n\uppi\over4}\right)-\sin{n\uppi\over4}\,\sin\!\left(t\lambda-{n\uppi\over4}\right)\label{e5.5}\\
&=c_1\,(\sinh t)^{n/2}\lambda^{n/2}\varphi_\lambda(t)+c_2\,(\sinh t)^{n/2} \lambda^{n/2-1}\varphi_\lambda'(t)\notag\\
&\qquad+\underbrace{\e^{\i t\lambda}\mathfrak{a}_{1,t}(t\lambda)+\e^{-\i t\lambda}\mathfrak{a}_{2,t}(t\lambda)}_{\mathfrak{r}_t(\lambda)}\notag
\end{align}
with $\mathfrak{a}_{1,t}(\lambda), \mathfrak{a}_{2,t}(\lambda)\in S^{-1}$ uniformly in $0<t\leq 1$. In addition, from \cite[(4.7)]{WaYa} we have
\begin{align}\label{e5.6}
\|(1-\eta(t\sqrt{L}))\,m(\sqrt{L})\,\mathfrak{r}_t(\sqrt{L})\|_{L^1(G)\to L^1(G)}\leq C.
\end{align}

Now we consider the two terms $\lambda^{n/2}\varphi_\lambda(t)$ and $\lambda^{n/2-1}\varphi_\lambda'(t)$ in \eqref{e5.5}. If we set
\begin{align*}
\widetilde{\mathfrak{m}}_{1,t}(\lambda):=\mathfrak{m}_{1,t}(\lambda)\,\varphi_\lambda(t),\quad \widetilde{\mathfrak{m}}_{2,t}(\lambda):=\mathfrak{m}_{2,t}(\lambda)\,\varphi'_\lambda(t)
\end{align*}
with
\begin{align*}
\mathfrak{m}_{1,t}(\lambda):=(1-\eta(t\lambda))\, m(\lambda)\,\lambda^{n/2}\in S^0
\end{align*}
and
\begin{align*}
\mathfrak{m}_{2,t}(\lambda):=(1-\eta(t\lambda))\, m(\lambda)\,\lambda^{n/2-1}\in S^{-1}
\end{align*}
uniformly in $0<t\leq 1$, in view of \eqref{e5.5} and \eqref{e5.6}, the proof of \eqref{e5.4} reduces to showing that
\begin{align}\label{e5.7}
\|\widetilde{\mathfrak{m}}_{\ell,t}(\sqrt{L})\|_{L^1(G)\to L^{1,\infty}(G)}\leq C,\quad \ell=1,2.
\end{align} 

We now show \eqref{e5.7} for $\ell=1$. In view of Lemma \ref{lem2.6}, the operator $\mathfrak{m}_{1,t}(\sqrt{L})$ is bounded from $L^1(G)$ to $L^{1,\infty}(G)$ uniformly in $0<t\leq 1$. In addition, by Lemmas \ref{lem4.1} and \ref{lem2.1}, we have
\begin{align*}
\|\varphi_{\sqrt{L}}(t)\|_{L^1(G)\to L^1(G)}\leq C\int_{\SS^n}\delta^{-1/2}(g(t,\omega))\,\d\omega\leq C.
\end{align*}
Therefore,
\begin{align*}
\|\widetilde{\mathfrak{m}}_{1,t}(\sqrt{L})\|_{L^1(G)\to L^{1,\infty}(G)}&\leq \|\mathfrak{m}_{1,t}(\sqrt{L})\|_{L^1(G)\to L^{1,\infty}(G)}\\
&\qquad\times\|\varphi_{\sqrt{L}}(t)\|_{L^1(G)\to L^1(G)}\leq C,
\end{align*}
and \eqref{e5.7} holds for $\ell=1$.

To verify \eqref{e5.7} for $\ell=2$, in analogy to \eqref{e3.11.3}, we rewrite
\begin{align}\label{e5.7.1}
\widetilde{\mathfrak{m}}_{2,t}(\sqrt{L})=\mathfrak{m}_{2,t}(\sqrt{L})\,A_t^*+\mathfrak{m}_{2,t}(\sqrt{L})X^*\,B_{0,t}^*+\sum_{j=1}^n\mathfrak{m}_{2,t}(\sqrt{L})Y_j^*\,B_{j,t}^*.
\end{align}
This, together with \eqref{e3.11.2}, Lemmas \ref{lem2.6} and \ref{thm3.2}, yields that
\begin{align*}
\|\widetilde{\mathfrak{m}}_{2,t}(\sqrt{L})\|_{L^1(G)\to L^{1,\infty}(G)}\leq C,
\end{align*}
and \eqref{e5.7} holds for $\ell=2$. From \eqref{e5.7} we conclude the proof of (\ref{prop5.2(a)}) for $0<t\leq 1$ in \textbf{Case} \ref{thm3.3case2}. This ends the proof of (\ref{prop5.2(a)}).
\end{asparaenum}
\smallskip

The result (\ref{prop5.2(b)}) can be proved by a similar argument, and we skip the detail here. The proof of Proposition \ref{prop5.2} is concluded.
\end{proof}
\smallskip

Next we show the sharpness of the regularities indices in Theorem \ref{thm1.1}. In other words, we will show the following result.
\begin{proposition}
Let $t\in \R\setminus\{0\}$ and $\alpha\in\R$.
\begin{asparaenum}[\rm (a)]
\item \label{thm3.1(a)}If the operator $(\Id+L)^{-\alpha/2}\cos(t\sqrt{L})$ is of weak-type $(1,1)$, then we must have $\alpha\geq n/2$.
\item \label{thm3.1(b)}If the operator $(\Id+L)^{-\alpha/2}\sin(t\sqrt{L})/\sqrt{L}$ is of weak-type $(1,1)$, then we must have $\alpha\geq n/2-1$.
\end{asparaenum}
\end{proposition}
\begin{proof}
Here we only show (\ref{thm3.1(a)}), since (\ref{thm3.1(b)}) follows analogously.

We assume contrarily that $(\Id+L)^{-\alpha/2}\cos(t\sqrt{L})$ is of weak-type $(1,1)$ for some $\alpha=n/2-\varepsilon$ and $\varepsilon>0$. In view of Lemma \ref{lem3.3} and classical estimates on Fourier transforms of symbols \cite[p. 241]{St}, $(\Id+L)^{-\gamma}$ has an $L^1$-integrable convolution kernel whenever $\gamma>0$. Hence we may assume $\varepsilon<1/10$. Then it follows from interpolation with the trivial $L^2$-boundedness of $(\Id+L)^{-n/4+\varepsilon/2}\cos(t\sqrt{L})$ that this operator is bounded on $L^p(G)$ for all $1<p\leq 2$, in particular for $p$ such that $100n/\varepsilon<p'<\infty$. But this contradicts with the necessity part of \cite[Theorem 1.1]{WaYa}.
\end{proof}

Finally, let us show that the polynomial growth $1+|t|$ in the estimate \eqref{s1e5} is best possible. Without losing any generality, we fix $t>0$. To simplify calculations, we let $s=\log x$ and adopt $(s,y)$-coordinates. This time
\begin{align*}
\d\rho(s,y)=\d s\d y,\qquad L=-X^2-\sum_{j=1}^nY_j^2=-\partial_s^2-\e^{2s}\Delta_y
\end{align*}
for $\Delta_y$ being the standard Laplacian on $\R^n$, while the Riemannian metric \eqref{riem} becomes
\begin{align}\label{riem-sy}
\g=\d s^2+\e^{-2s}\d y^2.
\end{align}

We let $\eta\in C^\infty_c([-2,2])$ with $0\leq q_t\leq 1$ and $\eta\equiv 1$ on $[-1,1]$, and define
\begin{align*}
q_t(s,y):=\eta\!\left({s\over 2t+1}+3\right)\prod_{j=1}^n\eta(y_j).
\end{align*}
Then
\begin{align*}
\supp q_t\subset \widetilde{\mathcal{T}}_t:=\{-10t-5\leq s\leq -2t-1,\, |y_j|\leq 2\}
\end{align*}
and
\begin{align*}
q_t\equiv 1\quad\text{on }\mathcal{T}_t:=\{-8t-4\leq s\leq -4t-2,\, |y_j|\leq 1\}.
\end{align*}
\begin{lemma}\label{lem6.1}
For all $\beta\in\R$,
\begin{align*}
\|(\Id+L)^\beta(q_t)\|_{L^1(G)}\leq C_\beta\, (1+t).
\end{align*}
\end{lemma}
\begin{proof}
Let $m=\max\{\lceil\beta\rceil,0\}$. Then
\begin{align}\label{eq:6.2}
\|(\Id+L)^\beta(q_t)\|_{L^1(G)}\leq \underbrace{\|(\Id+L)^{\beta-m}\|_{L^1(G)\to L^1(G)}}_{\mathbf I}\underbrace{\|(\Id+L)^m(q_t)\|_{L^1(G)}}_{\mathbf J}.
\end{align}
Since $(\Id+L)^m$ is a linear conbination of the following differential operators:
\begin{align*}
\e^{2ks}\partial_s^{2\ell}\Delta_y^k,\qquad k+\ell\leq m,
\end{align*}
and on $\supp q_t\subset \widetilde{\mathcal{T}}_t$
\begin{align*}
\e^{2ks}\partial_s^{2\ell}\Delta_y^k(q_T)\leq C_{k,\ell},
\end{align*}
We have
\begin{align*}
\mathbf J\leq C_m\,\rho(\widetilde{\mathcal{T}}_t)\leq C_\beta\, (1+t).
\end{align*}

On the other hand, if $\beta=0,1,2,\dots$ we have $\mathbf I=1$; otherwise by \eqref{e3.5}, Lemma \ref{lem3.3} and classical estimates on Fourier transforms
of symbols \cite[p. 241]{St}, $(\Id+L)^{\beta-m}$ has an $L^1$-integrable convolution kernel, whence $\mathbf I\leq C$. In view of \eqref{eq:6.2}, we obtain the desired estimate.
\end{proof}
In view of Lemma \ref{lem6.1}, the sharpness of the polynomial growth $1+|t|$ in the estimate \eqref{s1e5} follows readily from the following result.

\begin{proposition}\label{prop4.6}
Suppose $t>0$. 
\begin{asparaenum}[\rm (i)]
\item \label{prop6.2(i)}$\|\cos(t\sqrt{L})(q_t)\|_{L^{1,\infty}(G)}\geq C\, (1+t)$; 
\item \label{prop6.2(ii)}$\big\|{\sin(t\sqrt{L})\over\sqrt{L}}(q_t)\big\|_{L^{1,\infty}(G)}\geq Ct\,(1+t)$.
\end{asparaenum}
\end{proposition}

\begin{proof}
By the intertwining relationship \eqref{e3.3} and the argument showing \cite[Theorem 2.4.2]{Sog14}, we know that the wave equation \eqref{s1e3} has \textit{unit speed propagation}. In other words, for
$$
(s,y)\in S_t:=\{-7t-4\leq s\leq -5t-2,\, |y_j|\leq 1/2\},
$$
the functions
$$
\cos(t\sqrt{L})(q_t)(s,y)\quad\text{and}\quad{\sin(t\sqrt{L})\over\sqrt{L}}(q_t)(s,y)
$$
depend only on $q_t$ restricted on the $t$-neighborhood $\mathcal N_t(S_t)$ of $S_t$.

We claim that $\mathcal N_t(S_t) \subset \mathcal T_t$. Indeed, suppose $(s,y)\in S_t$, and $(s(\tau),y(\tau))$ be any arc-lengthed geodesic ray starting from $(s,y)$. By \eqref{riem-sy}, this means that
\begin{align*}
|s'(\tau)|^2+\e^{-2s(\tau)}\sum_{j=1}^n|y_j'(\tau)|^2\equiv 1.
\end{align*}
Then for $0\leq t_0\leq t$,
\begin{align}\label{e6.3}
|s(t_0)-s|\leq \int_0^{t_0}|s'(\tau)|\,\d\tau\leq t,
\end{align}
and in view of \eqref{e6.3} and the fact that $s\leq -5t-2$,
\begin{align*}
|y_j(t_0)-y_j|\leq \int_0^{t_0} |y_j'(\tau)|\,\d \tau\leq t\sup_{0\leq\tau\leq t}\e^{s(\tau)}\leq t\,\e^{-4t-4}\leq {\e^{-5}\over 4}.
\end{align*}
Hence for $0\leq \tau\leq t$,
\begin{align*}
-8t-4\leq s(\tau)\leq -4t-2,\qquad |y_j(\tau)|\leq 1,
\end{align*}
whence $(s(\tau),y(\tau))\in \mathcal{T}_t$. The claim is verified.

To show Proposition \ref{prop4.6}\eqref{prop6.2(i)}, we note that $u\equiv 1$ is the solution to the Cauchy problem
\begin{align*}
{\partial^2u\over\partial t^2}+Lu=0,\quad u|_{t=0}\equiv 1,\quad\left.{\partial u\over\partial t}\right|_{t=0}\equiv0.
\end{align*}
As a consequence, by the unit speed propagation property and uniqueness of the solution (cf. \cite[pp. 12--13]{Sog14}), one has
\begin{align*}
\cos(t\sqrt{L})(q_t)\equiv 1 \qquad\text{on }S_t.
\end{align*}
Hence
\begin{align*}
\|\cos(t\sqrt{L})(q_t)\|_{L^{1,\infty}(G)}\geq {\rho(S_t)\over 2}\geq C\,(1+t).
\end{align*}
This, combined with the fact that $\|q_t\|_{L^1(G)}\leq C\,(1+t)$, yields Proposition \ref{prop4.6}\eqref{prop6.2(i)}.

It remains to verify Proposition \ref{prop4.6}\eqref{prop6.2(ii)}. Note that this time $u(t,\cdot)\equiv t$ is the solution to the Cauchy problem
\begin{align*}
{\partial^2u\over\partial t^2}+Lu=0,\quad u|_{t=0}\equiv 0,\quad\left.{\partial u\over\partial t}\right|_{t=0}\equiv1.
\end{align*}
As a consequence, by the unit speed propagation property and uniqueness of the solution (cf. \cite[pp. 12--13]{Sog14}) again, one has
\begin{align*}
{\sin(t\sqrt{L})\over \sqrt{L}}(q_t)\equiv t \qquad\text{on }S_t.
\end{align*}
Hence
\begin{align*}
\left\|{\sin(t\sqrt{L})\over \sqrt{L}}(q_t)\right\|_{L^{1,\infty}(G)}\geq {t\over 2}\rho(S_t)\geq Ct\,(1+t).
\end{align*}
Using the fact that $\|q_t\|_{L^1(G)}\leq C\,(1+t)$, we show Proposition \ref{prop4.6}\eqref{prop6.2(ii)}. 
This ends the proof of Proposition~\ref{prop4.6}.

The proof of 	Theorem \ref{thm1.1} is complete.

\end{proof}

\smallskip\noindent
\textbf{Acknowledgements.} The authors were supported by National Key R$\&$D Program of China 2022YFA1005700 and 
by NNSF of China (No. 12571111).


\bibliographystyle{plain}

\end{document}